\documentclass[10pt,a4paper]{amsart}

\usepackage[latin1]{inputenc}
\usepackage{tikz}
\usepackage{caption}
\usepackage{amsmath}
\usepackage{amssymb}
\usepackage{graphics, float}
\usepackage{amsthm}
\usepackage{bm}
\usepackage{hyperref}
\hypersetup{colorlinks=true,citecolor=blue,linkcolor=blue}%
\usepackage{mathrsfs}
\usepackage{enumitem}
\usepackage{physics}
\DeclareMathAlphabet{\mathpzc}{OT1}{pzc}{m}{it}
\newtheorem{theorem}{Theorem}
\newtheorem{corollary}[theorem]{Corollary}
\newtheorem{lemma}[theorem]{Lemma}
\newtheorem{proposition}[theorem]{Proposition}

\theoremstyle{definition}
\newtheorem{definition}[theorem]{Definition}
\theoremstyle{remark}

\newcommand{\R}{\mathbb{R}}
\newcommand{\T}{\mathbb{T}}

\newcommand{\C}{\mathbb{C}}

\newcommand{\vertiii}[1]{{\left\vert\kern-0.25ex\left\vert\kern-0.25ex\left\vert #1 
    \right\vert\kern-0.25ex\right\vert\kern-0.25ex\right\vert}}

\begin{document}

	\title[]{\large Complex Analytic Structure of Stationary Flows of an Ideal Incompressible Fluid} 
	\author{Aleksander Danielski}	
	\author{Alexander Shnirelman}
		\address{Department of Mathematics and Statistics, Concordia University\\
		Montreal, QC H3G 1M8, Canada\\}
		\email{al\_danielski@hotmail.com, alexander.shnirelman@concordia.ca}

		\begin{abstract}
In this article we introduce the structure of an analytic Banach manifold in the set of stationary flows without stagnation points of the ideal incompressible fluid in a periodic 2-d channel bounded by the curves $y=f(x)$ and $y=g(x)$ where $f, g$ are periodic analytic functions. The work is based on the recent discovery (Serfati, Shnirelman, Frisch, and others) that for the stationary flows the level lines of the stream function (and hence the flow lines) are real-analytic curves. The set of such  functions is not a linear subspace of any reasonable function space. However, we are able to introduce in this set a structure of a real-analytic Banach manifold if we regard its elements as collections of level lines parametrized by the function value. If $\psi(x,y)$ is the stream function, then the flow line has equation $y=a(x,\psi)$ where $a(\cdot,\cdot)$ is a ``partially analytic" function. This means that this function is analytic in the first argument while it has a finite regularity in the second one. We define the spaces of analytic functions on the line, analogous to the Hardy space, and the spaces of partially analytic functions. The equation $\Delta\psi=F(\psi)$ is transformed into a quasilinear equation $\Phi(a)=F$ for the function $a(x,\psi)$. Using the Analytic Implicit Function Theorem in the complex Banach space, we are able to prove that  for the functions $f(x), g(x)$ close to constant the solution $a(x,\psi; F, f, g)$ exists and depends analytically on parameters $F, f, g$.
\end{abstract}
	
	\maketitle

{\large

\section{Stationary solutions of the 2D Euler equations}
\pagenumbering{arabic}

The Euler equations, describing the flow of an incompressible, inviscid fluid of uniform density were first published by Euler in 1757 (\cite{Eu}). In the absence of external forces, they take the form:
\begin{equation}\label{euler}
\pdv{\mathbf{u}}{t} + \mathbf{u} \cdot \nabla \mathbf{u} + \nabla p = 0,  \qquad \nabla \cdot \mathbf{u} = 0.
\end{equation}
Here, $\mathbf{u}(x,t)$ is the vector field describing the fluid velocity at any moment in time and $p(x,t)$ is a scalar field describing the pressure exerted on the fluid particle occupying position $x$ by the surrounding fluid. If the fluid occupies a domain with boundary, then additionally, $\mathbf{u}$ is required to be tangent to this boundary. Taking the divergence of the first equation, $\mathbf{u}$ is seen to define $p$ uniquely up to an additive constant by a Poisson equation. The usual problem then is to find $\mathbf{u}(x,t)$ given initial velocity $\mathbf{u}(x,0)$.

In three dimensions, local in time existence and uniqueness of classical solutions was proved in the 1920s by Lichtenstein (\cite{Li}) and Gyunter (\cite{Gu}). In two dimensions, Wolibner (\cite{Wo}) proved such solutions extend globally in time. More recently, it was discovered that despite finite regularity of the flow $\mathbf{u}$, the particle trajectories of solutions to \ref{euler} are real analytic curves. This striking fact was first proved by Serfati (\cite{Se}), later by Shnirelman (\cite{Sh1}) and others (\cite{Co}, \cite{FrZh}, \cite{InKaTo}, \cite{In}, \cite{Na}). In the case of stationary (time-independent) solutions, the particle trajectories coincide with flow lines (integral curves of $\mathbf{u}$) and thus the latter are real analytic curves as well.

This work is devoted to the simplest class of solutions to the 2D Euler equations, namely to the  stationary flows. Let us remind that any divergence-free vector field $\mathbf{u}$ can be defined uniquely (up to additive constant) by the stream function $\psi$, with $\mathbf{u} = \nabla^{\perp} \psi$. The flow lines of $\mathbf{u}$ coincide with the level lines of $\psi$ and stagnation points of $\mathbf{u}$ coincide with critical points of $\psi$. The vorticity $\omega = \nabla \times \mathbf{u}$ of a 2D vector field points normal to the flow, and can thus be taken as a scalar related to the stream function by the expression $\omega = \Delta \psi$. The 2D stationary solutions of \ref{euler} are those vector fields whose vorticities are constant along their flow lines (level lines of $\psi$). At least if $\psi$ is strictly monotone transversally to the flow lines, the stationary flows satisfy
\begin{equation}\label{stationary}
\Delta \psi = F(\psi), \qquad \psi \rvert_{\Gamma_i} = c_i,
\end{equation}
where $F(\psi)$ is the vorticity along the flow lines. The second condition indicates that $\psi$ should be constant along each component of the boundary $\Gamma$ of the domain the fluid occupies.

\v{S}ver\'{a}k and Choffrut (\cite{ChSv}) considered stationary flows in annular domains having no stagnation points. However, the stationary solutions do not form a smooth manifold in the space of divergence-free vector fields of finite regularity. To overcome this difficulty, the above authors considered the $C^\infty$-smooth vector fields, and they used the Nash-Moser-Hamilton implicit function theorem to provide a local parameterization of the set of stationary flows in the Fr\'{e}chet space of smooth vector fields. However, the analyticity of flow lines brings new possibilities to understand the structure of the set of stationary flows.

To this end we change the viewpoint and consider the flow field as a family of analytic flow lines non-analytically depending on parameter. We quantify the analyticity by introducing spaces of functions which have an analytic continuation to some complex strip containing the real axis such that on the boundary of the strip the function belongs to the Sobolev space. Further, we introduce the class of Sobolev functions of two variables which are analytic (in the above sense) with respect to one variable. Such functions describe the families of flow lines of stationary flows. These partially-analytic functions form a complex Banach space. The stationary solutions satisfy (in the new variables) a quasilinear elliptic equation whose local solvability is proved by using the analytic implicit function theorem in complex Banach spces. Thus we prove that the set of stationary flows is an analytic manifold in the complex Banach space of the flows (i.e. families of flow lines). In this work, this general idea is realized for the case of stationary flows in a periodic channel having no stagnation points.


\section{Analytic flow lines and the Sobolev space of partially-analytic functions}

Let us consider a periodic channel bounded by the graphs of two analytic functions. The lower and upper boundaries are defined by $y=f(x)$ and $y=g(x)$, respectively. Consider a stationary flow in the said channel for which $\psi=0$ along the lower boundary  $y=f(x)$ and $\psi=1$ along the upper one $y=g(x)$. Such a flow satisfies equation:
\begin{equation}\label{semilinear}
\Delta \psi(x,y) = F(\psi), \qquad \psi(x,f(x)) = 0, \qquad \psi(x,g(x)) = 1,
\end{equation}
for $x \in \T$, $f(x) \leq y \leq g(x)$. A typical example is the constant parallel flow, defined by $\psi = y$, with $f(x) = 0$, $g(x)=1$ and $F(\psi) = 0$.

If there are no stagnation points in the flow, we expect that each flow line is also the graph of some function in $(x,y)$ coordinates. In such a case, the collection of flow lines is monotonically parameterized by values of the stream function $\psi \in [0,1]$. We denote the family of such flow lines by function $y=a(x,\psi)$, with $(x,\psi) \in \Pi = \T \times [0,1]$. This coordinate change was introduced by von Mises in 1927 in his work on boundary layers (\cite{Mi}), and by Dubreil-Jacotin in her 1934 work on free surface waves (\cite{Du}). Since Barron's 1989 (\cite{Ba}) use of the coordinate change for the numerical study of flows over airfoils, it has seen numerous applications in computational problems, where it is known as the computational von Mises transform (see \cite{Ha} for a survey). Its success is owed in part to the fact that it converts complicated domains of flow to rectangular `computational' domains in $(x,\psi)$ coordinates.

Inverting the Jacobian of this transformation, one finds expressions for $\psi_x$ and $\psi_y$ in terms of $a(x,\psi)$. In particular, the velocity field in the new coordinates is given by
\begin{equation}
\mathbf{u}(\psi,\theta) =(\psi_y, -\psi_x) = \frac{(1,a_x)}{a_\psi}.
\end{equation}
We see then that stagnation points occur only when $a_\psi$ grows unbounded. Using the above relation, an application of the chain rule gives the following expression for the vorticity in our new coordinates:
\begin{equation}\label{quasilinear}
\Delta \psi = \Phi(a) = \frac{-1}{a_\psi}a_{xx} + \frac{2a_x}{a_\psi^2}a_{x \psi} - \frac{1+a_x^2}{a_\psi^3}a_{\psi \psi}.
\end{equation}
$\Phi(a)$ is a second order quasilinear operator of form $Aa_{xx} + 2Ba_{x\psi} + Ca_{\psi \psi}$. Such an operator is elliptic if $AC-B^2>0$. We immediately see that since $AC-B^2 =1/a_\psi^4$, the operator $\Phi$ is elliptic away from any stagnation points of the flow (where the ellipticity degenerates).

We conclude that under the coordinate change $\psi(x,y) \to y=a(x,\psi)$, the equation of stationary flow in the periodic channel without fixed points transforms to the quasilinear elliptic boundary value problem:
\begin{equation}\label{NLBVP}
\Phi(a) = F(\psi), \qquad a(x,0) = f(x), \qquad a(x,1) = g(x),
\end{equation}
to be solved for $y=a(x,\psi)$, given boundary data $f(x)$ and $g(x)$ and vorticity $F(\psi)$. Here $(x,\psi)$ is taken in the periodic strip $\Pi = \T \times [0,1]$. 

Solving this equation is the main objective of our study and we should start by defining the function spaces appropriate for this goal. Let us briefly discuss what properties such a space should have. First, we remind that $y=a(x,\psi)$ represents the family of analytic flow lines used to describe our flows. That is, for each fixed $\psi$, $x \to a(x,\psi)$ parameterizes the given flow line and thus $a(x,\psi)$ should be analytic with respect to $x$. Secondly, $\psi \to a(\cdot,\psi)$ parameterizes the family of such flow lines, and thus the dependence on $\psi$ is of finite regularity. For this reason, we call such functions partially-analytic. Having said that, the natural setting for posing elliptic boundary value problems are the usual Sobolev spaces. The isotropy in their regularity scale is a crucial element for establishing well-posedness of such equations. We can conclude that the candidate space for functions $a(x,\psi)$ should possess an isotropic Sobolev structure modelled on $H^m(\Pi)$ and a secondary regularity scale to quantify the analyticity in $x$.

Let us first see how this analyticity is incorporated to form a space for the individual flow lines, that is, analytic functions $y=a(x)$, $x \in \T$. Any real analytic function can be analytically continued to some complex domain containing the real axis. We consider those functions $y=a(x)$ which can be continued to the complex strip $\T_\sigma = \T \times i(-\sigma, \sigma)$, for some $\sigma > 0$. One of the Paley-Wiener theorems (a well known set of results relating the decay of a function's Fourier transform to its analytic extensions) characterizes the Hardy space of holomorphic functions in the complex strip. While the original result (\cite{PaWi}) considers continuations of $L^2(\R)$ functions, it adapts easily to continuations of $H^m(\T)$ functions with minimal modification.

\begin{definition}
Let $a(x)$ be a function on the circle $\T$ with Fourier series $a(x) =  \sum_k \hat{a}_k e^{ikx}$. Define $X_\sigma^m(\T)$ to be the space of such functions admitting analytic continuations to the complex strip $\T_\sigma$ such that their restrictions to the boundaries are $H^m(\T)$ functions. The Paley-Wiener theorem gives two equivalent characterizations of this space:
	\begin{itemize}
		\item $X_\sigma^m(\T)$ is the space of functions $a(x)$ on the circle with norm
			\begin{equation*}
			\lVert a(x) \rVert_{X_\sigma^m}^2 = \sum_k (1+k^2)^m e^{2\sigma \lvert k \rvert} \lvert \hat{a}_k \rvert^2 < \infty.
			\end{equation*}
		\item $X_\sigma^m(\T)$ is the space of analytic functions $z= x+it \to a(z): \T_\sigma \to \C$ with norm
			\begin{equation*}
			\lVert a(z) \rVert_{X_\sigma^m}^2 = \lVert a(\cdot + i \sigma) \rVert_{H^m(\T)}^2 + \lVert a(\cdot - i \sigma) \rVert_{H^m(\T)}^2 < \infty.		
			\end{equation*}
	\end{itemize}
\end{definition}

$X_\sigma^m(\T)$ defines the space of individual complex analytic flow lines. Apart from the domain of analyticity, this space also specifies the behaviour of their possible complex singularities.

Next we define the subset of $H^m(\Pi)$ functions which are partially-analytic, that is, analytic in $x$. The restrictions of such functions to vertical sections defines a map $a \in H^m(\Pi) \to a(x,\cdot) \in H^m[0,1]$. The additional smoothness in $x$ prevents the typical half-order loss of regularity in this trace map. The Paley-Wiener theorem can be extended to such Banach-valued maps to characterize the Hardy space of partially analytic Sobolev functions.

\begin{definition}
Let $a(x,\psi)$ be a function on the periodic strip $\Pi$ with partial Fourier series $a(x,\psi) = \sum_k \hat{a}_k(\psi) e^{ikx}$. Define $Y_\sigma^m(\Pi)$ to be the space of such functions which are partially analytic in $x$ in the following equivalent sense:
	\begin{itemize}
		\item $Y_\sigma^m(\Pi)$ is the space of functions $a(x,\psi)$ on the periodic strip with norm
			\begin{equation*}
			\lVert a(x,\psi) \rVert_{Y_\sigma^m}^2 = \sum_{p+q=0}^m \sum_k k^{2q} e^{2\sigma \lvert k \rvert} \lVert D^p \hat{a}_k(\psi) \rVert_{L^2[0,1]}^2 < \infty.
			\end{equation*}
		\item $Y_\sigma^m(\Pi)$ is the space of analytic functions $z \to a(z,\cdot) : \T_\sigma \to H^m[0,1]$ with norm
			\begin{equation*}
			\lVert a(z,\psi) \rVert_{Y_\sigma^m}^2 = \lVert a(\cdot + i \sigma, \cdot) \rVert_{H^m(\Pi)}^2 + \lVert a(\cdot - i \sigma, \cdot) \rVert_{H^m(\Pi)}^2 < \infty.		
			\end{equation*}
	\end{itemize}
\end{definition}

The complex Banach space $Y_\sigma^m(\Pi)$ defines the families of complex analytic flow lines which make up the stationary flows in our formulation. This space has an isotropic Sobolev structure, modelled on $H^m(\Pi)$, appropriate for the posing of elliptic boundary value problems. Most of our usual intuition of Sobolev functions extends to this space of partially-analytic functions. In particular, derivatives $\partial_x$ and $\partial_\psi$ are bounded in $Y_\sigma^m \to Y_\sigma^{m-1}$, as are the restriction to individual flow lines $a \to a(\cdot, \psi) : Y_\sigma^m(\Pi) \to X_\sigma^{m-1/2}(\T)$. Also of importance, the space $Y_\sigma^m$ is an algebra for $m > 1$. In the following section, we apply the implicit function theorem in Banach spaces to establish local solutions of equation \ref{NLBVP} in our defined spaces.


\section{Implicit function theorem in complex Banach spaces}

Define the operator
\begin{equation}\label{operator}
T : (f,g,F,a) \to \big( \Phi(a) - F(\psi), a(x,0) - f(x), a(x,1) - g(x) \big)
\end{equation}
in the spaces
\begin{equation*}
X_\sigma^{m-1/2} \times X_\sigma^{m-1/2}\times H^{m-2}[0,1] \times Y_\sigma^m \to Y_\sigma^{m-2} \times X_\sigma^{m-1/2} \times X_\sigma^{m-1/2}.
\end{equation*}
The quasilinear boundary value problem \ref{NLBVP} for the complex flow lines can be expressed by the operator equation $T=0$. The space $X_\sigma^{m-1/2}$ describes the individual flow lines from which we prescribe the boundaries of the channel, $y=f(x)$ and $y=g(x)$. The prescribed complex vorticities $F(\psi)$ are taken in the Sobolev space $H^{m-2}[0,1]$. Finally the families of flow lines $y=a(x,\psi)$ used to describe our flows are taken in $Y_\sigma^m$. 

This equation has particular solution for data $f=0$, $g=1$ and $F=0$ given by $a= \psi$, corresponding to the constant parallel flow. That is $T(0,1,0,\psi) = 0$. The implicit function theorem gives condition for local solvability of the equation near this solution.

\begin{theorem}[Analytic implicit function theorem in complex Banach space]\ \\
Let $X,Y,Z$ be complex Banach spaces and $f: X \times Y \to Z$ be analytic in a neighbourhood of $(x_0,y_0) \in X \times Y$. Suppose $f(x_0,y_0) = 0$ and $\pdv{f}{y} (x_0,y_0) : Y \to Z$ is an isomorphism. Then there exists a neighbourhood of $(x_0,y_0) \in X \times Y$ in which the equation $f(x,y)=0$ has a unique solution, parameterized by an analytic function $y=g(x) : X \to Y$.
\end{theorem}

Application of the implicit function theorem consists of two parts. First, we must prove that near the constant parallel flow, the map $T$ is analytic in our spaces. Second, we must show that the linearization at the constant parallel flow defines an isomorphism. A straightforward computation shows that $\Phi' = - \Delta$. That is, the linearization of the quasilinear map $\Phi(a)$ at $a = \psi$ is the Laplace operator on $\Pi$. Then the linearization of map $T$ defines the operator
\begin{equation}
\pdv{T}{a} (0,1,0, \psi) : a(x,\psi) \to \Big( -\Delta a(x,\psi) , a(x,0), a(x,1) \Big)
\end{equation}
and we must verify it defines an isomorphism in
\begin{equation*}
Y_\sigma^m(\Pi) \to Y_\sigma^{m-2}(\Pi) \times X_\sigma^{m-1/2}(\T) \times X_\sigma^{m-1/2}(\T).
\end{equation*}


\subsection{Nonlinear results}\ \\

Let us start by proving operator $T$ is analytic. The nonlinear map $\Phi(a)$ defined in \ref{quasilinear} can be viewed as the composition of derivative maps $a \to (a_x, a_\psi, a_{xx}, a_{x\psi}, a_{\psi \psi})$ and the rational function $\Phi(a_x, a_\psi, a_{xx}, a_{x\psi}, a_{\psi \psi})$. The latter, a superposition operator, the simplest of which are maps $u(x) \to f(u(x))$, are well behaved in Sobolev spaces of sufficient regularity. 

\begin{lemma}\ \\
Let $\Omega\subset \R^n$ be a bounded domain with smooth boundary and suppose $f \in C^m$ in a domain containing the image of $u$. Then $u(x) \to f(u(x)) : H^m(\Omega) \to H^m(\Omega)$ is a well-defined, continuous map for $m > n/2$. If additionally, $f \in C^{m+l}$, then $u \to f(u)$ is $C^l$.
\end{lemma}

\begin{proof}
The proof uses the Sobolev embeddings and H\"{o}lder's inequality to bound the $L^2$ norms of the products appearing in the expressions for derivatives of $f(u(x))$.
\end{proof}

\begin{lemma}\ \\
Suppose $f$ is complex analytic in a domain containing the image of $a(z,\psi)$. Then $a(z,\psi) \to f(a(z,\psi)) : Y_\sigma^m(\Pi) \to Y_\sigma^m(\Pi)$ is an analytic map for $m > 1$.
\end{lemma}

\begin{proof}
The function $z \to f(a(z,\cdot)$ is the composition of analytic function $z \to a(z,\cdot) : \T_\sigma \to H^m[0,1]$ and the map $a(z,\cdot) \to f(a(z,\dot)) : H^m[0,1] \to H^m[0,1]$. By the previous lemma, the latter map is well defined for $m > 1/2$ and complex differentiable. By the theory of holomorphy in complex Banach spaces (\cite{Mu}), it is thus analytic. The function $z \to f(a(z,\cdot)) : \T_\sigma \to H^m[0,1]$ is thus analytic. Furthermore, for $m>1$, we have $a(\cdot \pm i\sigma, \cdot) \to f(a(\cdot \pm i\sigma, \cdot)) : H^m(\Pi) \to H^m(\Pi)$ and thus $\lVert f(a(\cdot \pm i\sigma, \cdot)) \rVert_{H^m(\Pi)} < \infty$. We conclude $a \to f(a) : Y_\sigma^m \to Y_\sigma^m$ is well defined. The derivative of this map corresponds to multiplication by $f'(a)$. That is, $Df(a): u \to f'(a)u$. By the above results $f'(a) \in Y_\sigma^m$. Since this space is an algebra for $m > 1$, the Fr\'{e}chet derivative $Df(a)$ is well defined in $Y_\sigma^m \to Y_\sigma^m$. Thus $a \to f(a) : Y_\sigma^m \to Y_\sigma^m$ is complex differentiable and therefore analytic.
\end{proof}

\begin{corollary}\label{nonlinear}\ \\
Suppose $m > 3$. In a sufficiently small neighbourhood of solution $a(z,\psi) = \psi \in Y_\sigma^m(\Pi)$, the operator $T$ defined in \ref{operator} is complex analytic.
\end{corollary}

\begin{proof}
The quasilinear map $\Phi(a)$ is the composition of bounded linear (and thus analytic) map
\begin{equation*}
a \to (a_x, a_\psi, a_{xx}, a_{x\psi}, a_{\psi \psi}) : Y_\sigma^m \to Y_\sigma^{m-2} \times \cdot \cdot \cdot \times Y_\sigma^{m-2}
\end{equation*}
and the superposition with a rational function
\begin{equation*}
(a_x, a_\psi, a_{xx}, a_{x\psi}, a_{\psi \psi}) \to \Phi(a_x, a_\psi, a_{xx}, a_{x\psi}, a_{\psi \psi}) : Y_\sigma^{m-2} \times \cdot \cdot \cdot \times Y_\sigma^{m-2}  \to Y_\sigma^{m-2}.
\end{equation*}
By the preceding lemma, this superposition operator is analytic for $m>3$ so long as the denominator is never zero, that is, when $a_\psi \neq 0$. Suppose $\lVert a - \psi \rVert_{Y_\sigma^m} < \varepsilon$. From the embedding $H^m(\Pi) \subset C^1(\Pi)$ for $m > 2$, we have
\begin{equation*}
\sup_{\lvert t \rvert < \sigma} \lVert a_\psi(\cdot + it, \cdot) - 1 \rVert_{\infty} \leq C \sup_{\lvert t \rvert < \sigma} \lVert a(\cdot + it, \cdot) - \psi \rVert_{H^m(\Pi)} \leq C \lVert a - \psi \rVert_{Y_\sigma^m}.
\end{equation*}
Taking $\varepsilon$ sufficiently small gives $a_\psi$ sufficiently close to $1$ that it never vanishes and by consequence, $a \to \Phi(a)$ is analytic.

Next, the identity map defines an embedding $F(\psi) \to F(z,\psi) : H^{m-2}[0,1] \to Y_\sigma^{m-2}$. This map is linear, bounded and thus analytic. The same is true for the restrictions maps $a(z,\psi) \to a(z,0)$ and $a(z,\psi)  \to a(z,1) : Y_\sigma^m \to X_\sigma^{m-1/2}$. The map $(f,g,F,a) \to T$ is thus analytic in each variable and thus an analytic operator in the spaces $X_\sigma^{m-1/2} \times X_\sigma^{m-1/2}\times H^{m-2}[0,1] \times Y_\sigma^m \to Y_\sigma^{m-2} \times X_\sigma^{m-1/2} \times X_\sigma^{m-1/2}$. 
\end{proof}


\subsection{Linearization}

\begin{proposition}\label{linear} \ \\
The linearization of $T$ with respect to $a(z,\psi)$ at the constant parallel flow, given by operator
\begin{equation*}
a(z,\psi) \rightarrow (-\Delta a(z,\psi), a(z,0), a(z,1)) : Y_\sigma^m \rightarrow Y_\sigma^{m-2} \times X_\sigma^{m-1/2} \times X_\sigma^{m-1/2},
\end{equation*}
defines a Banach space isomorphism.
\end{proposition}

\begin{proof}
That this operator is bounded follows immediately from boundedness of derivatives $\partial_x, \partial_\psi : Y_\sigma^m \to Y_\sigma^{m-1}$ and restrictions $a \to a(\cdot, \psi): Y_\sigma^m \to X_\sigma^{m-1/2}$. To show that this operator is invertible, we must solve the Dirichlet problem for the Poisson equation on the domain $\Omega = \T \times [0,1]$ :
\begin{equation*}
\Delta a(x,\psi) = f(x,\psi), \qquad a(x,0) = b(x), \qquad a(x,1) = c(x).
\end{equation*}
Expanding all functions as Fourier series in $x$, we get the family of ODEs
\begin{equation*}
\hat{a}''_k(\psi) - k^2 \hat{a}_k(\psi) = \hat{f}_k(\psi), \qquad \hat{a}_k(0) = \hat{b}_k, \qquad \hat{a}_k(1) = \hat{c}_k.
\end{equation*}
We construct the solution as follows. In the case when $k=0$, straight integration yields
\begin{equation*}
\hat{a}_0(\psi) = \int_0^\psi \int_0^\eta \hat{f}_0(t) dt d\eta + \left(\hat{c}_0 - \hat{b}_0 - \int_0^1 \int_0^\eta \hat{f}_0(t) dt d\eta\right)\psi + \hat{b}_0.
\end{equation*}
For general $k$, let's write the solution as $\hat{a}_k(\psi) = \hat{h}_k(\psi) + \hat{g}_k(\psi)$ where $\hat{h}_k$ is the solution to the homogeneous part (with non-homogeneous boundary conditions) and $\hat{g}_k$ is the general part (with homogeneous boundary conditions). Then we get
\begin{equation*}
\hat{h}_k(\psi) = \frac{\hat{c}_k\sinh(k\psi)-\hat{b}_k\sinh[k(\psi-1)]}{\sinh(k)},
\end{equation*}
and
\begin{equation*}
\hat{g}_k(\psi) = \int_0^1 G_k(\psi,t) \hat{f}_k(t) dt,
\end{equation*}
where
\begin{equation*} G_k(\psi,t) = 
\begin{cases}  \dfrac{\sinh[k(t-1)]\sinh(k\psi) }{k \sinh(k)} \qquad  \ \ \ \psi \leq t \\ \\
\dfrac{\sinh(kt)\sinh[k(\psi-1)] }{k \sinh(k)} \qquad \psi \geq t
\end{cases}
\end{equation*}
is the Green's function. The solution to the Dirichlet problem is given by 
\begin{equation*}
a(x,\psi) = h(x,\psi) + g(x,\psi) = \sum_k \hat{h}_k(\psi)e^{ikx} + \sum_k \hat{g}_k(\psi)e^{ikx}.
\end{equation*}

Next we show that this solution belongs to  $Y_\sigma^m$. That is,
\begin{equation*}
\|a\|_{Y_\sigma^m}^2 = \sum_{p+q\leq m} \sum_k e^{2\sigma|k|} \left(k^2\right)^q \| D^p \hat{a}_k(\psi) \|_{L^2(0,1)}^2 < \infty.
\end{equation*}
First let's check $h(x,\psi)$. We have
\begin{equation*}
D^{p} \hat{h}_k(\psi)= \frac{k^{p}}{\sinh(k)} \left(\hat{c}_k\sinh(k\psi)- k^{p} \hat{b}_k\sinh[k(\psi-1)] \right)
\end{equation*}
when $p$ is even and
\begin{equation*}
D^{p} \hat{h}_k(\psi)= \frac{k^{p}}{\sinh(k)} \left( \hat{c}_k\cosh(k\psi)-k^{p} \hat{b}_k\cosh[k(\psi-1)] \right)
\end{equation*}
when $p$ is odd. Notice that,
\begin{equation*}
\int_0^1 \sinh^2(k\psi) d\psi = \int_0^1 \sinh^2[k(\psi-1)] d\psi = \frac{\sinh(k)\cosh(k) - k}{2k}
\end{equation*}
and
\begin{equation*}
\int_0^1 \cosh^2(k\psi) d\psi = \int_0^1 \cosh^2[k(\psi-1)] d\psi = \frac{\sinh(k)\cosh(k) + k}{2k}.
\end{equation*}
Therefore we get
\begin{equation*}
\| D^{p} \hat{h}_k(\psi)\|_{L^2(0,1)}^2 \leq (k^2)^{p} (\hat{b}_k^2 + \hat{c}_k^2) \frac{\sinh(k)\cosh(k) \pm k}{2k\sinh^2(k)} \approx \frac{(k^2)^{p}}{k} (\hat{b}_k^2 + \hat{c}_k^2).
\end{equation*}
It follows that 
\begin{equation*}
e^{2\sigma|k|}(k^2)^{q} \|D^p \hat{h}_k(\psi)\|_{L^2(0,1)}^2 \approx  e^{2\sigma|k|}(k^2)^{p + q - 1/2 }(\hat{b}_k^2 + \hat{c}_k^2) \approx e^{2\sigma|k|}(1+k^2)^{m - 1/2 }(\hat{b}_k^2 + \hat{c}_k^2)
\end{equation*}
since $p + q \leq m$. Summing over $k, p, q$ we get
\begin{equation*}
\|h(x,\psi)\|_{Y_\sigma^m}^2 \leq C \left(\|b(x)\|_{X_\sigma^{m-1/2}}^2 + \|c(x)\|_{X_\sigma^{m-1/2}}^2\right).
\end{equation*}

Now let's consider the term $g(x,\psi)$. Starting with the equation
\begin{equation*}
\hat{g}_k'' = k^2\hat{g}_k + \hat{f}_k,
\end{equation*}
differentiating an even number of times gives
\begin{equation*}
D^{2n}\hat{g}_k = k^{2n}\hat{g}_k + \sum_{j=0}^{n-1} k^{2(n-1-j)} D^{2j}\hat{f}_k
\end{equation*}
and differentiating an odd number of times gives
\begin{equation*}
D^{2n+1}\hat{g}_k = k^{2n} \hat{g}_k' + \sum_{j=0}^{n-1} k^{2(n-1-j)} D^{2j+1}\hat{f}_k.
\end{equation*}
To control the $L^2$ norm of $D^p \hat{g}_k$, we need to estimate the norms of $\hat{g}_k$ and $\hat{g}_k'$.
Recall $\hat{g}_k$ satisfies 
\begin{equation*}
\begin{cases} \hat{g}_k'' - k^2\hat{g}_k = \hat{f}_k \\ \hat{g}_k(0) = \hat{g}_k(1) = 0
\end{cases}
\end{equation*}
on the interval $[0,1]$. Multiplying both sides of the ODE by $\overline{\hat{g}}_k$ and integrating over the interval gives
\begin{equation*}
\int_0^1 \hat{g}_k'' \overline{\hat{g}}_k d\psi - k^2\int_0^1 \hat{g}_k \overline{\hat{g}}_k d\psi = \int_0^1 \hat{f}_k \overline{\hat{g}}_k d\psi.
\end{equation*}
Integrating by parts and using that $\hat{g}_k$ vanishes on the end points, we get
\begin{equation*}
\int_0^1 |\hat{g}_k'|^2 d\psi + k^2 \int_0^1 |\hat{g}_k|^2 d\psi = - \int_0^1 \hat{f}_k \overline{\hat{g}}_k d\psi.
\end{equation*}
Applying Poincar\'{e} inequality:
\begin{equation*}
\int_0^1 |\hat{g}_k|^2 d\psi \leq C\int_0^1 |\hat{g}_k'|^2 d\psi
\end{equation*}
and Cauchy-Schwarz inequality, we get
\begin{equation*}
(1+k^2) \|\hat{g}_k\|_{L^2(0,1)} \leq C\|\hat{f}_k\|_{L^2(0,1)}
\end{equation*}
and so
\begin{equation*}
\|\hat{g}_k\|_{L^2}^2 \leq C \frac{\|\hat{f}_k\|_{L^2}^2}{k^4}.
\end{equation*}
Also notice we have
\begin{equation*}
\int_0^1 |\hat{g}_k'|^2 d\psi \leq \int_0^1 |\hat{g}_k'|^2 + k^2 |\hat{g}_k|^2 d\psi \leq \left( \int_0^1 |\hat{f}_k|^2 d\psi \right)^{1/2} \left(\int_0^1 |\hat{g}_k|^2 d\psi \right)^{1/2}.
\end{equation*}
It follows that
\begin{equation*}
\|\hat{g}_k'\|_{L^2}^2 \leq \|\hat{f}_k\|_{L^2}  \|\hat{g}_k\|_{L^2} \leq C \frac{\|\hat{f}_k\|_{L^2}^2}{k^2}.
\end{equation*}
We now have control of $\hat{g}_k$ and $\hat{g}_k'$, which we can use to control higher derivatives:
\begin{equation*}
\| D^{2n}\hat{g}_k \|_{L^2}^2 \leq  C \sum_{j=0}^{n-1} (k^2)^{2n-2-2j} \| D^{2j}\hat{f}_k\|_{L^2}^2
\end{equation*}
and
\begin{equation*}
\| D^{2n+1}\hat{g}_k \|_{L^2}^2 \leq C(k^2)^{2n-1} \|\hat{f}_k\|_{L^2}^2 + C\sum_{j=0}^{n-1} (k^2)^{2n-2-2j} \| D^{2j+1}\hat{f}_k\|_{L^2}^2.
\end{equation*}
Combining the two cases, we write
\begin{equation*}
\| D^p \hat{g}_k\|_{L^2}^2 \leq \sum_{j=0}^{p-2} C_j \left( k^2 \right)^{p-2-j} \| D^j \hat{f}_k \|_{L^2}^2.
\end{equation*}
Multiplying both sides by $(k^2)^q$ and using that $p+q \leq m$ we get
\begin{equation*}
\left( k^2\right)^q \|D^p \hat{g}_k\|_{L^2}^2 \leq \sum_{j=0}^{m-2} C_j \left( k^2 \right)^{m-2-j} \|\hat{f}_k^{(j)}\|_{L^2}^2.
\end{equation*}
Finally, multiplying both sides by $e^{2\sigma|k|}$ and summing over $k,p,q$ we obtain
\begin{equation*}
\|g\|_{Y_\sigma^m}\leq C\|f\|_{Y_\sigma^{m-2}}.
\end{equation*}
Combining bounds on the homogeneous and general part of the solution, we obtain $\|a\|_{Y_\sigma^m}^2 \leq C (\|f\|_{Y_\sigma^{m-2}}^2 + \|b\|_{X_\sigma^{m-1/2}}^2 + \|c\|_{X_\sigma^{m-1/2}}^2)$ and conclude the bounded inverse map to the linearization exists.
\end{proof}

\section{Complex analytic structure of stationary flow lines}

By the implicit function theorem, the results \ref{nonlinear} and \ref{linear} of the previous section prove our main result:

\begin{theorem}\ \\
Suppose $m > 3$. Then in a neighbourhood of
\begin{equation*}
(f,g,F,a) = (0,1,0,\psi) \in X_\sigma^{m-1/2} \times X_\sigma^{m-1/2} \times H^{m-2}[0,1] \times Y_\sigma^m, 
\end{equation*}
equation \ref{NLBVP} has locally a unique solution
\begin{equation*}
(f,g,F) \to a : X_\sigma^{m-1/2} \times X_\sigma^{m-1/2} \times H^{m-2}[0,1] \to Y_\sigma^m,
\end{equation*}
depending analytically on its parameters.
\end{theorem}

Our result proves that near the constant parallel flow, the stationary flows in a periodic channel having no stagnation points form an analytic Banach manifold in the space of complex analytic flow lines. The theorem demonstrates that not only is analyticity of intermediate flow lines inherited from the boundary, the behaviour of their complex singularities is as well.

The general principle behind this work is the description of a stationary flow as a collection of its analytic flow lines. In this work, we have realized this picture for flows without stagnation points. In our following work, we extend the results to flows having a single elliptic stagnation point. New adaptions of spaces of partially-analytic functions must be introduced to accommodate the degeneracy of the equations at the stagnation point and the singular nature of solutions in our formulation.

Our description opens the way to more challenging problems: parameterizations near general solutions having a single stagnation point or none, patching such solutions together to describe stationary flows with an arbitrary number of stagnation points, properties of transition maps between charts on the prospective manifold of stationary flows, etc. Finally, can the spirit of this work be extended to accommodate other classes of flows, for instance, time-periodic solutions? One can envision giving systems of flow lines room to `breath' to produce such solutions.

}
\bibliographystyle{amsplain}
\providecommand{\bysame}{\leavevmode\hbox to3em{\hrulefill}\thinspace}
\providecommand{\MR}{\relax\ifhmode\unskip\space\fi MR }
\providecommand{\MRhref}[2]{%
  \href{http://www.ams.org/mathscinet-getitem?mr=#1}{#2}
}
\providecommand{\href}[2]{#2}

\end{document}